\documentclass[12pt,reqno]{amsart}
\usepackage[left=3cm,top=2cm,right=3cm,bottom=2cm]{geometry}
\usepackage{ltxcmds}
\usepackage{pdftexcmds}
\usepackage{amssymb}
\usepackage{amsbsy}
\usepackage{epsfig}
\usepackage{tikz}
\usepackage{verbatim}
\usepackage[normalem]{ulem}
\usepackage[pdftex]{hyperref}
\hypersetup{colorlinks=true,linkcolor=blue,citecolor=blue,breaklinks = true}

\tikzstyle{vertex} = [fill,shape=circle,node distance=80pt]
\tikzstyle{edge} = [fill,opacity=.5,fill opacity=.5,line cap=round, line join=round, line width=40pt]
\tikzstyle{elabel} =  [fill,shape=circle,node distance=30pt]

\pgfdeclarelayer{background}
\pgfsetlayers{background,main}

\begin{document}
\title{Adjacency energy of Hypergraphs}	

\date{\today}

\author[K. Cardoso]{Kau\^e Cardoso} \address{Instituto Federal do Rio Grande do Sul - Campus Feliz, CEP 95770-000, Feliz, RS, Brazil}\email{\tt kaue.cardoso@feliz.ifrs.edu.br}

\author[R.R. Del-Vecchio]{Renata Del-Vecchio} \address{Instituto de Matem\'atica e Estat\'{\i}stica, UFF, CEP 24210-200, Niter\'oi, RJ, Brazil}\email{\tt rrdelvecchio@id.uff.br}
 	
\author[L.L.S. Portugal]{Lucas Portugal} \address{Graduate program in Mathematics, UFF, CEP 24210-200, Niter\'oi, RJ, Brazil}\email{\tt 	lucasportugal@id.uff.br}

\author[V. Trevisan]{Vilmar Trevisan} \address{Instituto de Matem\'atica e Estat\'{\i}stica, UFRGS,  CEP 91509-900, Porto Alegre, RS, Brazil and \\
Department of Mathematics and Applications, University of Naples Federico II, Italy} \email{\tt trevisan@mat.ufrgs.br}

\pdfpagewidth 8.5 in \pdfpageheight 11 in

\newcommand{\h}{\mathcal{H}}
\newcommand{\g}{\mathcal{G}}
\newcommand{\s}{\mathcal{S}}
\newcommand{\A}{\mathbf{A}}
\newcommand{\Q}{\mathbf{Q}}
\newcommand{\B}{\mathbf{B}}
\newcommand{\C}{\mathbf{C}}
\newcommand{\D}{\mathbf{D}}
\newcommand{\M}{\mathbf{M}}
\newcommand{\N}{\mathbf{N}}
\newcommand{\E}{\mathcal{E}}
\newcommand{\x}{\mathbf{x}}
\newcommand{\y}{\mathbf{y}}
\newcommand{\z}{\mathbf{z}}

\theoremstyle{plain}
\newtheorem{Teo}{Theorem}
\newtheorem{Lem}[Teo]{Lemma}
\newtheorem{Pro}[Teo]{Proposition}
\newtheorem{Cor}[Teo]{Corollary}

\theoremstyle{definition}
\newtheorem{Def}{Definition}[section]
\newtheorem{Afi}[Def]{Affirmation}
\newtheorem{Que}[Def]{Question}
\newtheorem{Exe}[Def]{Example}
\newtheorem{Obs}[Def]{Remark}

\maketitle

\begin{abstract} In this paper, we define and obtain several properties of the (adjacency) energy of a hypergraph. In particular, bounds for this energy are obtained as functions of structural and spectral parameters, such as Zagreb index and spectral radius. We also study how the energy of a hypergraph varies when a vertex/edge is removed or when an edge is divied. In addition, we solved the extremal problem energy for the class of hyperstars, and show that the energy of a hypergraph is never an odd number.\newline


\end{abstract}

\section{Introduction}
In Spectral graph theory, the structure of graphs is studied through the eigenvalues/eigenvectors of matrices associated with them. Many researchers around the world, motivated by this theory, have defined some matrices associated with hypergraph, aiming to develop a spectral hypergraph theory \cite{Feng,Reff2014,Rodriguez1}. In 2012, Cooper and Dutle \cite{Cooper} proposed the study of hypergraphs by means of  the adjacency tensor. It is known, however, that to obtain  eigenvalues of tensors has a high computational and theoretical cost \cite{NP-hard}. Perhaps for this reason, recently, some authors have renewed the interest to study matrix representations of hypergraphs, as for example in \cite{Banerjee-matriz,Reff2019,kr-regular,matrix-pec-radius-2017,estrada-index,distance,extremal-hypertree}.

The study of molecular orbital energy levels of $\pi$-electrons in conjugated hydrocarbons may be seen as one of the oldest applications of spectral graph theory \cite{energy-book,energy-huckel}. In those studies, graphs were used to represent hydrocarbon molecules and it was shown that an approximation of the total $\pi$-electron energy may be  computed from the eigenvalues of the graph. Based on this chemical concept, in 1977 Gutman \cite{energy-Gutman} defined graph energy. In 2007, Nikiforov \cite{energy-niki} extended the concept of graph energy to matrices. For a matrix $\M$, its \textit{energy} $\E(\M)$, is defined as the sum of its singular values.

Let $\h$ be a hypergraph and $\A$ its adjacency matrix. We define here the energy of $\h$ as $\E(\A)$ and denote it by  $\E(\h)$.

The main purpose of this paper is to discuss this natural definition of (adjacency)  energy of a hypergraph and to understand whether known properties of graph energy may be extended to the structure of hypergraphs.

We highlight some of the results we obtain. We start by studying a particular class of hypergraphs and determine which hyperstar has highest and lowest energy among this class. By understanding the operations \emph{sum} and \emph{product} of hypergraphs, we obtain restrictions on the values that $\E(\h)$ may have. In particular, we show that  $\E(\h)$ is never an odd number. In addition, we determine bounds for the variation of the energy of a hypergraph $\h$, when we perform operations on its edges or vertices. Our main contribution is to obtain several lower and upper bounds relating $\E(\h)$ to important structural and spectral parameters.

The remaining of the paper is organized as follows. In Section \ref{sec:pre},
we recall some basic definitions about hypergraphs and matrices that will be used. In Section \ref{star}, we study the energy of a hyperstar. In Section \ref{odd}, we prove that the energy of a hypergraph can never be an odd number. In sections \ref{deletion} and \ref{division}, we study what happens when we delete or divide an edge of a hypergraph. Finally, in Section \ref{bounds}, we will obtain bounds for the energy of a hypergrap. These bounds are related to important spectral and structural parameters, such as Zagreb index and spectral radius.

\section{Preliminaries}\label{sec:pre}
In this section, we shall present some basic definitions about hypergraphs and matrices, as well as terminology, notation and concepts that will be useful in our proofs.
\vspace{0.5cm}

A \textit{hypergraph} $\h=(V,E)$ is a pair composed by a set of vertices $V(\h)$ and a set of (hyper)edges $E(\h)$, where each edge is a subset of $V(\h)$, with cardinality greater than or equal $2$.  The \textit{rank} and the \textit{co-rank} of a hypergraph are defined as the largest and smallest cardinality of its edges, respectively. $\h$ is said to be a $k$-\textit{uniform hypergraph} (or a \textit{$k$-graph}) for $k \geq 2$, if all edges have the same cardinality $k$. Let $\mathcal{H}=(V,E)$ and $\mathcal{H}'=(V',E')$ be hypergraphs, if $V'\subseteq V$ and $E'\subseteq E$, then $\mathcal{H}'$ is a \textit{subgraph} of $\h$. The \textit{complete $k$-graph} $\mathcal{K}^{[k]}_n$ on $n$ vertices, is a hypergraph, such that any subset of $k$ vertices is an edge. A hypergraph $\h$ is \textit{linear} if each pair of edges has at most one common vertex.

A \textit{multi-hypergraph} $\h=(V,E)$  is a pair composed by a set of vertices $V(\h)$ and a multi-set of (multi-)edges $E(\h)$, where each edge is a subset of $V(\h)$.

Notice that in a hypergraph the edges must have at least two vertices and distinct edges cannot have exactly the same vertices, while in a multi-hypergraph there may be edges containing exactly the same vertices. Additionally, there could be edges with one or zero vertices.

Let $\h=(V,E)$ be a (multi-)hypergraph. The \textit{edge neighborhood} of a vertex $v\in V$, denoted by $E_{[v]}$, is the set of all edges that contains $v$. The \textit{degree} of a vertex $v\in V$, denoted by $d(v)$, is the number of edges that contain $v$. More precisely, $\;d(v) = |E_{[v]}|$. A hypergraph is $r$-\textit{regular} if $d(v) = r$ for all $v \in V$. We define the \textit{maximum}, \textit{minimum} and \textit{average} degrees, respectively, as
\[\Delta(\h) = \max_{v \in V}\{d(v)\}, \quad \delta(\h) = \min_{v \in V}\{d(v)\}, \quad d(\h) = \frac{1}{|V|}\sum_{v \in V}d(v).\]
Let $\alpha=\{v_ {i_1},\ldots, v_ {i_r}\} \subset V$ be a subset of vertices. We define the \textit{degree of a set} $\alpha$, as the number of edges that contain simultaneously all vertices of $\alpha$ and denote it for $ d(\alpha)$.

Let $\h$ be a hypergraph. A \textit{walk} of length $l$ is a sequence of
vertices and edges $v_0e_1v_1e_2 \ldots e_lv_l$ where $v_{i-1}$ and $v_i$ are
distinct vertices contained in $e_i$ for each $ i=1,\ldots,l$. A \textit{cycle} is a walk where $v_0=v_l.$ The hypergraph is \textit{connected}, if for each pair of vertices $ u, w$ there is a walk
$v_0e_1v_1e_2 \cdots e_lv_l $ where $ u = v_0 $ and $ w = v_l $. Otherwise,
the hypergraph is \textit{disconnected}.

Let $\h = (V, E)$ be a (multi-)hypergraph with $n$ vertices. The \textit{adjacency matrix} $\A(\h)=(a_{ij})$ is defined as a square matrix with $n$ rows. For each pair of vertices $i,j \in V$, if $i=j$, then $a_{ii}=0$, and if $i\neq j$, then $a_{ij}=d(\{i,j\})$. We denote its \textit{characteristic polynomial} by $P_\A(\lambda)=\det(\lambda\mathbf{I}_n-\A)$. Its eigenvalues will be denoted by $\lambda_1(\A)\geq\cdots\geq\lambda_n(\A)$.  If $\x$ is an eigenvector of $\lambda$, we call the pair $(\lambda,\x)$, as an \textit{eigenpair} of $\A$. The \textit{spectral radius} $\rho(\A)$, is the largest absolute value of its eigenvalues. The \textit{(adjacency) energy} of a hypergraph is the sum of the singular values of its adjacency matrix. We notice that $\A$ is a real and symmetric square matrix, then its energy is the sum of the absolute values of the eigenvalues, in other words $$\E(\h) = \sum_{i=1}^n |\lambda_i|.$$

Let $\h=(V,E)$ be a hypergraph with $n$ vertices. For a non-empty subset of vertices $\alpha = \{v_1,\ldots,v_t\} \subset V$ and a vector $\x=(x_i)$ of dimension $n$, we denote $x(\alpha)=x_{v_1}+\cdots+x_{v_t}$. So we have,
\[(\A(\h)\x)_u = \sum_{e \in E_{[u]}}x\left(e-\{u\}\right), \quad \forall u \in V(\h).\]

\section{The energy of a hyperstar}\label{star}
In this section, we will obtain the energy of a hyperstar and determe which hyperstar with $n$ vertices has the highest energy.
\vspace{0.5cm}

First of all, we will present the definition of a power graph, and obtain an algebraic property of its eigenvalues.
\begin{Def}
	Let $\g=(V,E)$ be a graph and let $k \geq 2$ be an
	integer. We define the \textit{power graph} $\g^k$ as
	the $k$-graph with the following sets of vertices and edges
	$$V(\g^k)=V(\g)\cup\left( \bigcup_{e\in E(\g)} \varsigma_e\right) \;\; \textrm{and}\;\;
	E(\g^k)=\{e\cup\varsigma_e \colon e \in E(\g)\},$$	where  $\varsigma_e=\{v_1^e,
	\ldots,v_{k-2}^e\}$ for each edge $e \in E(\g)$.
\end{Def}

We may say that $\g^k$ is obtained from a \textit{base graph} $\g=(V,E)$, adding $k-2$ new vertices to each edge $e \in E(\g)$. For each edge $e\in E(\g)$, we denote $e^k=e\cup \varsigma_e \in E(\g^k)$. The spectrum of this class has already been studied, see for example \cite{Kaue-powers,Kaue-lap}.

We define a \emph{hyperstar} as a power graph of a star.

\begin{Exe}
	The power graph $(S_4)^3$ of the star $S_4$ is illustrated in Figure \ref{fig:ex1}.
	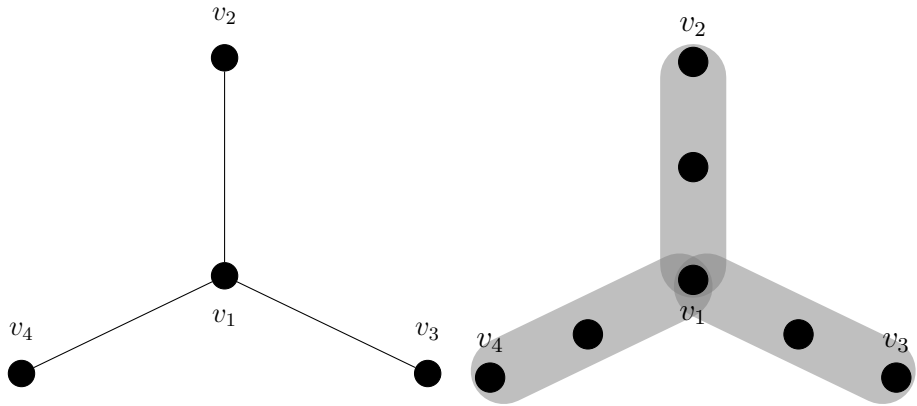
\begin{figure}[h]
		\centering
		\begin{tikzpicture}
		[scale=1,auto=left,every node/.style={circle,scale=0.9}]
		\node[draw,circle,fill=black,label=below:,label=below:\(v_1\)] (v1) at (0,0) {};
		\node[draw,circle,fill=black,label=below:,label=above:\(v_2\)] (v2) at (0,2.9) {};
		\node[draw,circle,fill=black,label=below:,label=above:\(v_3\)] (v3) at (2.7,-1.3) {};
		\node[draw,circle,fill=black,label=below:,label=above:\(v_4\)] (v4) at (-2.7,-1.3) {};
		\path
		(v1) edge node[left]{} (v2)
		(v1) edge node[below]{} (v3)
		(v1) edge node[left]{} (v4);
		\end{tikzpicture}
		\begin{tikzpicture}
		\node[draw,circle,fill=black,label=below:,label=below:\(v_{1}\)] (v11) at (0,0) {};
		
		\node[draw,circle,fill=black,label=below:,label=above:] (v22) at (0,1.5) {};
		\node[draw,circle,fill=black,label=below:,label=above:\(v_{2}\)] (v21) at (0,2.9) {};
		
		\node[draw,circle,fill=black,label=below:,label=above:] (v42) at (1.4,-0.725) {};
		\node[draw,circle,fill=black,label=below:,label=above:\(v_{3}\)] (v31) at (2.7,-1.3) {};
		
		\node[draw,circle,fill=black,label=below:,label=above:] (v62) at (-1.4,-0.725) {};
		\node[draw,circle,fill=black,label=below:,label=above:\(v_{4}\)] (v41) at (-2.7,-1.3) {};

		\begin{pgfonlayer}{background}
		\draw[edge,color=gray,line width=25pt] (v11) -- (v21);
		\draw[edge,color=gray,line width=25pt]  (v11) -- (v31);
		\draw[edge,color=gray,line width=25pt]  (v11) -- (v41);
		\end{pgfonlayer}
		\end{tikzpicture}
		\caption{The power hypergraph $(S_4)^3$.}\label{fig:ex1}
	\end{figure}
\end{Exe}

\begin{Lem}\label{lem:vertex-same-edge}
	Let $\h$ be a $k$-graph having two vertices $u$ and $v$ which are contained exactly in the same edges. If $(\lambda, \x)$ is an eigenpair of $\A(\h)$ with
	$\lambda \neq -d(u)$, then $x_u = x_v$.
\end{Lem}
\begin{proof}
	We observe that,
	\[(\lambda+d(u)) x_u = \sum_{e \in E_{[u]}}x(e) = \sum_{e \in E_{[v]}}x(e) =
	(\lambda+d(v)) x_v.\]
	Since $d(u) = d(v)$ and $\lambda \neq -d(u) $, then the result follows.
\end{proof}

\begin{Teo}
	Let $S_{n}$ be the star on $n$ vertices. If $k \geq 2$ is an integer, then spectrum of $(S_{n})^k$ is:
	$$\{(-1)^{(n-1)(k-2)},\;(k-2)^{n-2},\;r^+,\;r^-\},$$
	where $r^+$ and $r^-$ are the roots of $x^2 - (k-2)x - (n-1)(k-1) = 0$.
\end{Teo}
\begin{proof}
	Let $e \in E(S_n)$ be an edge, let $\{u_1,\ldots,u_{k-1}\}$ be the vertices of degree one in $e^k$ and $2 \leq i \leq k-1$. We can construct the following family of $k-2$ linearly independent vectors.
	$$\x^i = \begin{cases}(x^i)_{u_1} =\;\;\, 1,\\
	(x^i)_{u_i} = -1,\\
	(x^i)_{u}\; =\;\;\, 0, \textrm{ for } u \in
	V((S_n)^k)-\{u_1,u_i\}.\end{cases}$$ Repeating this construction for the other
	edges of $S_n$, we obtain $(n-1)(k-2)$ linearly independent eigenvectors,
	associated with $\lambda = -1$.
	
	Let $E(S_n) = \{e_1,\ldots,e_{n-1}\}$ and $2 \leq j \leq n-1$. We can construct the following family of $n-2$ linearly independent vectors.
	$$\z^j = \begin{cases}(z^j)_{u} =\;\;\, 1, \textrm{ if } u \textrm{ is a vertex of degree one in  } (e_1)^k,\\
	(z^j)_{v} = -1, \textrm{ if } v \textrm{ is a vertex of degree one in } (e_j)^k,\\
	(z^j)_{w}\; =\;\, 0, \textrm{ if } w \textrm{ is not a vertex of degree one in } (e_1)^k \textrm{ or in } (e_j)^k.\end{cases}$$
	So, we have $n-2$ linearly independent eigenvectors, associated with $\lambda = k-2$.
	
	Let $(\lambda,\x)$ an eigenpair of $\A((S_n)^k)$, let $E(S_n) = \{e_1,\ldots,e_{n-1}\}$, let $\{u^j_1,\ldots,u^j_{k-1}\}$ be the vertices of degree one in $(e_j)^k$, and let $v$ be the vertex of degree greater than one in $(S_n)^k$. By Lemma \ref{lem:vertex-same-edge} we know that $x_{u^j_i} = x_{u^j_1}$ for all $i=2,\ldots,k-1$, so the system of eigenvalues can be writen as:
	
	$$\begin{cases}
	\lambda x_{u_1^1} = x_{v} + (k-2)x_{u_1^1}\\
	\lambda x_{u_1^2} = x_{v} + (k-2)x_{u_1^2}\\
	\vdots\\
	\lambda x_{u_1^m} = x_{v} + (k-2)x_{u_1^{n-1}}\\
	\lambda x_{v} = (k-1)x_{u_1^1}+\cdots+(k-1)x_{u_1^{n-1}}\\
	\end{cases}$$
	Notice that, if $\lambda \neq k-2$, then $x_{u_1^1}=\cdots=x_{u_1^{n-1}}$, so the system can be write as:
	
	$$\begin{cases}
	\lambda x_{u_1^1} = x_{v} + (k-2)x_{u_1^1}\\
	\lambda x_{v} = (n-1)(k-1)x_{u_1^1}\\
	\end{cases}$$
	By the first equation we have that, $x_{v} = (\lambda-k+2)x_{u_1^1}$, so by second equation we conclude that
	$$\lambda(\lambda-k+2)x_{u_1^1} = (n-1)(k-1)x_{u_1^1},$$
	so
	$$ \lambda^2 + \lambda(-k+2) - (n-1)(k-1)=0.$$
	
	Finally, to conclude the result just notice that the sum of the multiplicities of the eigenvalues is $|V((S_n)^k)|=(n-1)(k-1)+1$.
\end{proof}

\begin{Cor}
	Let $S_{n}$ be the star on $n$ vertices. If $k \geq 2$ is an integer, then
	$$\E((S_{n})^k) = (k-2)(2n-3) + \sqrt{(k-2)^2+4(n-1)(k-1)}.$$
\end{Cor}
\begin{proof}
	We just notice that
	$$(n-1)(k-2)|-1| + (n-2)|k-2| = (k-2)(2n-3).$$
	
	Now, observe that
	$$r^+ = \frac{(k-2) + \sqrt{(k-2)^2+4(n-1)(k-1)}}{2} \geq 0$$
	and
	$$r^- = \frac{(k-2) - \sqrt{(k-2)^2+4(n-1)(k-1)}}{2} \leq 0,$$
	so $|r^+| + |r^-| = r^+-r^- = \sqrt{(k-2)^2+4(n-1)(k-1)}.$
\end{proof}

\begin{Cor}
	If $\s$ is a hyperstar with $t$ vertices, then $$2\sqrt{t-1}=\E(S_t)\leq\E(\s)\leq\E((S_2)^t)=2(t-1).$$
\end{Cor}
\begin{proof}
	If $\s$ is a hyperstar, then there is $2 \leq n \leq t$ and $2 \leq k \leq t$ such that $\s = (S_n)^k$. In this way, we have that $t = (n-1)(k-1)+1$, so $n = \frac{t-1}{k-1}+1$, therefore
	$$\E(\s) = 2(t-1)\left(\frac{k-2}{k-1}\right) -(k-2)  + \sqrt{(k-2)^2+4(t-1)}.$$
	
	Consider the function $f:[2,t]\rightarrow\mathbb{R},$ defined by
	$$f(x) = 2(t-1)\left(\frac{x-2}{x-1}\right)-(x-2) + \sqrt{(x-2)^2+4(t-1)}.$$
	
	Computing its derivatives, we obtain
	\begin{eqnarray}
	f'(x) &=& \frac{2(t-1)}{(x-1)^2} -1 + \frac{x-2}{\sqrt{(x-2)^2+4(t-1)}},\notag\\
	f''(x) &=& \frac{-4(t-1)}{(x-1)^3} + \frac{4(t-1)}{((x-2)^2+4(t-1))^\frac{3}{2}}.\notag
	\end{eqnarray}
	
	We observe that $f''(x) < 0 \;\Leftrightarrow\; (x-1)^2 < (x-2)^2+4(t-1) \;\Leftrightarrow\; 2x+1 < 4t,$ therefore $f'(x)$ is a decreasing function.	Now, we notice that $f'(t) = \frac{2}{t-1} -\frac{2}{t} > 0$, so $f'(x) > 0$ for all $x \in [2,t]$, so $f(x)$ is an increasing function and therefore the result follows.
\end{proof}

\section{The energy of a hypergraph is never odd}\label{odd}
In this section, we will prove that the energy of a hypergraph can never be an odd number. More precisely, we prove that the energy of a hypergraph can not be a $p$-th root of an odd number, and if it is a $p$-root of an even number then its even prime factors have power smaller than $p$. Similar results have already been obtained for graphs in \cite{energy-odd1, energy-odd2}.
\vspace{0.5cm}

We start by defining some operations between hypergraphs, and computing the spectrum of the resulting hypergraph.

\begin{Def}Let $\h$ and $\g$ be $k$-graphs. We define its \textit{sum}
	$\h\oplus\g$ as the $k$-graph, with the sets of vertices $\;V(\h\oplus\g) =
	V(\h)\times V(\g)$ and edges	$$\;E(\h\oplus\g) = \{\{v\}\times e:\; v \in
	V(\h),\; e \in E(\g)\}\cup \{f\times \{u\}:\; \; f \in E(\h), u \in V(\g)\}.$$
\end{Def}

\begin{Exe}
     Let $\h$ and $\g$ be $3$-graphs, such that
	$V(\h)=\{1,2,3,4\}$, $E(\h)=\{123, 234\}$ and $V(\g)=\{a,b,c\}$, $E(\g)=\{abc\}$. The sum $\h\oplus\g$ has the following sets of vertices and edges
	$$V(\h\oplus\g)=\{(1,a), (1,b), (1,c), (2,a), (2,b), (2,c), (3,a), (3,b), (3,c), (4,a), (4,b), (4,c)\}.$$
	$$E(\h\oplus\g)=\left\lbrace\begin{matrix}
	\{(1,a), (1,b), (1,c)\}, & \{(1,a), (2,a), (3,a)\},\\
	\{(2,a), (2,b), (2,c)\}, & \{(1,b), (2,b), (3,b)\},\\
	\{(3,a), (3,b), (3,c)\}, & \{(1,c), (2,c), (3,c)\},\\
	\{(4,a), (4,b), (4,c)\}, & \{(2,a), (3,a), (4,a)\},\\
	\{(2,b), (3,b), (4,b)\}, & \{(2,c), (3,c), (4,c)\}
\end{matrix}\right\rbrace_. $$
\end{Exe}
	
\begin{Pro}\label{prop:sum}
	If $\g$ and $\h$ are two $k$-graphs having eigenvalues
	$\mu$ and $\lambda$ with multiplicities $m_1$ and $m_2$, respectively, then $\mu+\lambda$ is an eigenvalue of $\A(\g\oplus\h)$, with
	multiplicity $m_1\cdot m_2$.
\end{Pro}
\begin{proof}
	Suppose $\x$ an eigenvector of $\lambda$ in $\A(\h)$ and  $\y$ an eigenvector
	of $\mu$ in $\A(\g)$. Consider $(v,u)$ a vertex of $\g\oplus\h$, define a
	vector $\mathbf{z}$ by $z_{(v,u)} = y_vx_u$. Thus,
	\[(\A\mathbf{z})_{(v,u)}=\!\!\!\sum_{\alpha \in E_{[(v,u)]}}\!\!\!z(\alpha - (v,u)) = \sum_{e \in E_{[u]}}y_vx(e-u) + \sum_{a \in E_{[v]}}y(a-v)x_u  = (\mu+\lambda)y_vx_u.\]
	Therefore, the result follows.
\end{proof}
	
\begin{Def}
	Let $\h$ and $\g$ be $k$-graphs.
	\begin{enumerate}
		\item For each edge $e \in E(\h)$, we say that a sequence
		$\alpha=(v_1,v_2,\ldots,v_k)$ is an \textit{ordered edge} from $e$, if the set 	of its elements is equal to the edge $e$.
		
		\item For each edge $e \in E(\h)$, we denote by $S_\h(e)$ the set of all ordered edges from $e$.
		
		\item Let $e\in E(\h)$ and $f \in E(\g)$. For $\alpha=(v_1,\ldots,v_k) \in S_\h(e)$ and  $\beta=(u_1,\ldots,u_k) \in S_\g(f)$, we define its \textit{product} as the following set of ordered pairs $$\alpha\otimes\beta=\{(v_1,u_1),\ldots,(v_k,u_k)\}.$$
		
		\item We define the \textit{product} $\h\otimes\g$ as the $k$-graph, with the sets of vertices $\;V(\h\otimes\g) =  V(\h)\times V(\g)$ and edges
		$$\;E(\h\otimes\g) = \{\alpha\otimes\beta:\; \alpha\in S_\h(e), \textrm{ where } e\in E(\h), \textrm{ and } \beta\in S_\g(f), \textrm{ where } f\in E(\g).\}$$
	\end{enumerate}
\end{Def}

\begin{Exe} Let $\h$ and $\g$ be $3$-graphs, such that
	$V(\h)=\{1,2,3,4\}$, $E(\h)=\{123, 234\}$ and $V(\g)=\{a,b,c\}$, $E(\g)=\{abc\}$. The product $\h\otimes\g$ has the following sets of vertices and edges	
	$$V(\h\otimes\g)=\{(1,a), (1,b), (1,c), (2,a), (2,b), (2,c), (3,a), (3,b), (3,c), (4,a), (4,b), (4,c)\}.$$	
	$$E(\h\otimes\g)=\left\lbrace\begin{matrix}
	\{(1,a), (2,b), (3,c)\}, & \{(2,a), (3,b), (4,c)\},\\
	\{(1,a), (2,c), (3,b)\}, & \{(2,a), (3,c), (4,b)\},\\
	\{(1,b), (2,a), (3,c)\}, & \{(2,b), (3,a), (4,c)\},\\
	\{(1,b), (2,c), (3,a)\}, & \{(2,b), (3,c), (4,a)\},\\
	\{(1,c), (2,a), (3,b)\}, & \{(2,c), (3,a), (4,b)\},\\
	\{(1,c), (2,b), (3,a)\}, & \{(2,c), (3,b), (4,a)\}
\end{matrix}\right\rbrace_. $$
\end{Exe}	

\begin{Pro}\label{prop:prod}
	If $\g$ and $\h$ are two $k$-graphs, with eigenvalues
	$\mu$ of multiplicity $m_1$ and  $\lambda$ of multiplicity $m_2$
	respectively, then $\mu\lambda$ is an eigenvalue of $\A(\g\otimes\h)$, with
	multiplicity $m_1\cdot m_2$.
\end{Pro}
\begin{proof}
	Suppose $\x$ an eigenvector of $\lambda$ in $\A(\h)$ and  $\y$ an eigenvector
	of $\mu$ in $\A(\g)$. Consider $(v,u)$ a vertex of $\g\otimes\h$, define a
	vector $\mathbf{z}$ by $z_{(v,u)} = y_vx_u$. Thus,
	\[(\A\mathbf{z})_{(v,u)}=\!\!\!\sum_{\alpha \in E_{[(v,u)]}}\!\!\!z(\alpha - (v,u)) = \sum_{a \in E_{[v]}}\sum_{e \in E_{[u]}}y(a-v)x(e-u)  = \mu\lambda y_vx_u.\]
	Therefore, the result is true.
\end{proof}

\begin{Lem}\label{lem}
	Let $\h$ be a $k$-graph. If $\lambda_1,\ldots,\lambda_t$ are all positive eigenvalues of $\A(\h)$, then $\E(\h) = 2\sum_{i=1}^t\lambda_i$
\end{Lem}

\begin{proof}This follows, because the trace of the adjacency matrix is zero and equals the sum of the eigenvalues. \end{proof}

\begin{Obs}\label{obs}
	We notice that the characteristic polynomial of $\h$ is a monic polynomial with integer coefficients. Thus, if an eigenvalue of $\h$ is a rational number, then it has to be an integer.
\end{Obs}

\begin{Teo}\label{teo:even}
	Let $\h$ be a $k$-graph. If $\E(\h)$ is a rational number, then $\E(\h)$ is even.
\end{Teo}
\begin{proof}
	Let $\lambda_1,\ldots,\lambda_t$ be all positive eigenvalues of $\h$, so $\E(\h) = 2\sum_{i=1}^t\lambda_i$, that is $\sum_{i=1}^t\lambda_i$ is a rational number. By Proposition \ref{prop:sum}, we have that $\sum_{i=1}^t\lambda_i$ is an eigenvalue of a hypergraph and by Remark \ref{obs} we conclude that it must be an integer, therefore $\E(\h) = 2\sum_{i=1}^t\lambda_i$ must be an even number.	
\end{proof}

\begin{Teo}\label{teo:neverodd}
	Let $p$ and $q$ be integers such that $p \geq 1$ and $0 \leq q \leq p-1$ and $t$ be	an odd integer. If $\h$ is a $k$-graph, then $\E(\h) \neq \sqrt[p]{2^qt}$.
\end{Teo}
\begin{proof}
	Suppose by contradiction that $\E(\h) = \sqrt[p]{2^qt}$.  In this way, we have $$\left( \E(\h)\right)^p  = \left( 2\sum_{i=1}^t\lambda_i\right) ^p = 2^qt \Rightarrow \left( \sum_{i=1}^t\lambda_i\right) ^p = \frac{t}{2^{p-q}}.$$
	By Proposition \ref{prop:sum} and \ref{prop:prod}, we have that $\left(\sum_{i=1}^t\lambda_i\right)^p$ is an eigenvalue of a hypergraph, however it is a  non integral rational number, which is a contradiction.
\end{proof}

\section{Vertex and edge deletion}\label{deletion}
In this section we will study the variation of the energy of a hypergraph when we delete a vertex or an edge of the hypergraph.

\begin{Def}
	Let $\h=(V,E)$ be a hypergraph, $v\in V$ be a vertex and $e_1,\ldots,e_d$ be all edges containing $v$, we define $\h-v$ by $V(\h-v) =V(\h)-\{v\}$ and $$E(\h-v) = (E(\h)-\{e_1,\ldots,e_d\})\cup\{e_1-\{v\},\ldots,e_d-\{v\}\}.$$
\end{Def}

\begin{Teo}
	Let $\h$ be a hypergraph and $v \in V(\h)$ be a vertex, then
	$$\E(\h) \geq \E(\h-v).$$
\end{Teo}
\begin{proof}
	First notice that, $\A(\h-v)$ is a principal sub-matrix of $\A(\h)$. Let $\mu_1,\ldots,\mu_t$ be the positive eigenvalues of $\A(\h-v)$ and let $\mu_{t+1},\ldots,\mu_{n-1}$  be the non positive eigenvalues of $\A(\h-v)$. From the Interlace Theorem (see section 6.4 in \cite{linear-alg}), we have that
	$$\lambda_1(\h)\geq\mu_1, \quad\lambda_2(\h)\geq\mu_2,\cdots,\lambda_{t}(\h)\geq\mu_t.$$	
	$$|\lambda_n(\h)|\geq|\mu_{n-1}|, \quad|\lambda_{n-1}(\h)|\geq|\mu_{n-2}|,\cdots,|\lambda_{t+2}(\h)|\geq|\mu_{t+1}|.$$
	
	Therefore, $\E(\h) \geq \E(\h-v).$	
\end{proof}

Now, we will  bound the variation in energy when we delete an edge in the hypergraph.\\

The \textit{energy} of a matrix $\M$ is a generalization of the energy of graphs, it is computed as the sum of its singular values. In particular, if $\M$ is a real and symmetric square matrix, its energy is also the sum of the absolute values of the eigenvalues.

\begin{Lem}[Lemma 2.21, \cite{energy-lucelia}]\label{lem:energy-subgrafo}
	If $\M$ and $\mathbf{N}$ are square matrices, then
	$$\E(\M+\mathbf{N}) \leq \E(\M) + \E(\mathbf{N}), \quad |\E(\M) - \E(\mathbf{N})| \leq  \E(\M-\mathbf{N}). $$
\end{Lem}

\begin{Def}
	Let $\h=(V,E)$ be a hypergraph and $e\in E$ be an edge, we define $\h-e = (V, E-\{e\})$.
\end{Def}

\begin{Teo}
	Let $\h$ be a hypergraph and $e \in E(\h)$ be an edge, then
	$$|\E(\h) - \E(\h-e)| \leq 2|e|-2.$$
\end{Teo}
\begin{proof} First, we observe that
	\begin{eqnarray}
	|\E(\h) - \E(\h-e)| &\leq& \E\left(\A(\h) -  \A(\h-e)\right). \notag
	\end{eqnarray}
	\noindent The inequality above follows from Lemma \ref{lem:energy-subgrafo}. Now, we observe that
	
	\[\M :=
	\A(\h) -  \A(\h-e) =
	\left[
	\begin{array}{cccc|cccc}
	0 & 1 &\cdots&1&0&\cdots&0\\
	1 & 0 &\cdots&1&0&\cdots&0\\
	\vdots&\vdots&\ddots&\vdots& \vdots&\vdots&\vdots\\
	1 & 1 &\cdots& 0&0&\cdots& 0\\
	\hline
	0&0&\cdots&0& 0&\cdots&0\\
	\vdots&\vdots&\vdots&\vdots& \vdots&\ddots&\vdots\\
	0&0&\cdots&0& 0&\cdots&0\\
	\end{array}
	\right].
	\]
	That is, the eigenvalues of $\M$, are $|e|-1$ with multiplicity $1$ and $-1$ with multiplicity $|e|-1$, thus the energy of this matrix is
	$$\E \left(\A(\h) -  \A(\h-e)\right)  =  |e|-1 + (|e|-1)|-1| = 2|e|-2.$$
	Therefore, the result follows.		
\end{proof}


\begin{Exe}\label{ExAresta}Removing a vertex always implies that $\E(\h) \geq \E(\h-v).$ For edge deletion, this is no longer true as we see bellow.

	\begin{enumerate}
	    \item $\E(\h) > \E(\h-e)$ removing any edge $e\in E$:\\
		Let $\mathcal{K}_5^{[3]}$ be the complete $3$-uniform hypergraph with $5$ vertices and let $e \in E$. Notice that $\E(K_5^{[3]})=24, \quad \textrm{and} \quad\E(K_5^{[3]}-e)=21.731.$
	
	    \item $\E(\h) < \E(\h-e)$ removing any edge $e \in E$:\\
	    Let $\h=(V,E)$ be the hypergraph given by $V(\h)=\{1,2,3,4,5,6\}$ and $E=\{135, 136, 145, 146, 235, 236, 245, 246\}.$ We have that $\E(\h)=16.$
	    For $e \in E$, we have $\E(\h-e)=16,4926.$

	    \item $\E(\h)=\E(\h-e)$ removing removing a particular edge $e$:\\
	    Let $\h=(V,E)$ be the hypergraph give by two copies of $\mathcal{K}_5^{[3]}$ with $3$ distinct edges of cardinality $2$ with no common vertices connecting the two copies. Taking $e \in E$ as one of those edges, we have that  $\E(\h)=\E(\h-e)$.	
	\end{enumerate}	
\end{Exe}	

It is known that for graphs the same is true, that is, the energy of the graph when removing an edge, can increase, decrease or remain the same, \cite{energy-edge-deletion}.

\section{Edge division}\label{division}
Some important concepts of graph theory can be generalized in more than one way. In particular, removing edges of a graph can be generalized, in the context of hypergraphs, as the deletion of edges (seen in the previous section) or as the division of edges, that we will define in this section.

\begin{Def}
	Let $\h=(V,E)$ be a hypergraph and $F=\{e_1,\ldots,e_t\} \subset E$ a subset of edges. If we can divide each edge $e_i = e_i'\cup e_i''$ so that $|e_i'| > 0$, $|e_i''| > 0$ and $e_i'\cap e_i'' = \emptyset$, we define the following multi-hypergraph $$\h\triangleleft(F, F', F'') := \big(V,(E\cup F'\cup F'')-F \big),$$
	where $F' = \{e_1',\ldots,e_t'\}$ and $F'' = \{e_1'',\ldots,e_t''\}$. Under these conditions, we say that the multi-hypergraph $\h\triangleleft(F, F', F'')$ is obtained from $\h$ by dividing the edges in $F$. When is not necessary to explicitly say how the edges in $F$ are divided, we will denote simply by $\h\triangleleft F$. If $F$ is an unitary set, $F=\{e\}$, we denote by $\h\triangleleft e$.
\end{Def}

\begin{Exe} Consider the hypergraph $\h$ with the following vertices set $V(\h)=\{0,\ldots 9\}$ and edges  $E(\h)=\{012,234567,789\}$, in the Figure \ref{fig:ex3} we represent the hypergraph $\h\triangleleft(\{234567\},\{2345\},\{67\})$.
	\begin{figure}[h]
		\centering
		\begin{tikzpicture}
		[scale=1,auto=left,every node/.style={circle,scale=0.9}]
		\node[draw,circle,fill=black,label=below:,label=above:\(0\)] (v1) at (0,0) {};
		\node[draw,circle,fill=black,label=below:,label=above:\(1\)] (v2) at (1,0) {};
		\node[draw,circle,fill=black,label=below:,label=above:\(2\)] (v3) at (2,0) {};
		\node[draw,circle,fill=black,label=below:,label=above:\(3\)] (v4) at (3,0) {};
		\node[draw,circle,fill=black,label=below:,label=above:\(4\)] (v5) at (4,0) {};
		\node[draw,circle,fill=black,label=below:,label=above:\(5\)] (v6) at (5,0) {};
		\node[draw,circle,fill=black,label=below:,label=above:\(6\)] (v7) at (6,0) {};
		\node[draw,circle,fill=black,label=below:,label=above:\(7\)] (v8) at (7,0) {};
		\node[draw,circle,fill=black,label=below:,label=above:\(8\)] (v9) at (8,0) {};
		\node[draw,circle,fill=black,label=below:,label=above:\(9\)] (v10) at (9,0) {};
		
		\begin{pgfonlayer}{background}
		\draw[edge,color=gray,line width=40pt] (v1) -- (v3);
		\draw[edge,color=gray,line width=30pt]  (v3) -- (v8);
		\draw[edge,color=gray,line width=40pt]  (v8) -- (v10);
		\end{pgfonlayer}
		\end{tikzpicture}
		\begin{tikzpicture}
		[scale=1,auto=left,every node/.style={circle,scale=0.9}]
		\node[draw,circle,fill=black,label=below:,label=above:\(0\)] (v1) at (0,0) {};
		\node[draw,circle,fill=black,label=below:,label=above:\(1\)] (v2) at (1,0) {};
		\node[draw,circle,fill=black,label=below:,label=above:\(2\)] (v3) at (2,0) {};
		\node[draw,circle,fill=black,label=below:,label=above:\(3\)] (v4) at (3,0) {};
		\node[draw,circle,fill=black,label=below:,label=above:\(4\)] (v5) at (4,0) {};
		\node[draw,circle,fill=black,label=below:,label=above:\(5\)] (v6) at (5,0) {};
		\node[draw,circle,fill=black,label=below:,label=above:\(6\)] (v7) at (6,0) {};
		\node[draw,circle,fill=black,label=below:,label=above:\(7\)] (v8) at (7,0) {};
		\node[draw,circle,fill=black,label=below:,label=above:\(8\)] (v9) at (8,0) {};
		\node[draw,circle,fill=black,label=below:,label=above:\(9\)] (v10) at (9,0) {};
		
		\begin{pgfonlayer}{background}
		\draw[edge,color=gray,line width=40pt] (v1) -- (v3);
		\draw[edge,color=gray,line width=30pt]  (v3) -- (v6);
		\draw[edge,color=gray,line width=30pt]  (v7) -- (v8);
		\draw[edge,color=gray,line width=40pt]  (v8) -- (v10);
		\end{pgfonlayer}
		\end{tikzpicture}
		\caption{Edge division}
		\label{fig:ex3}
	\end{figure}
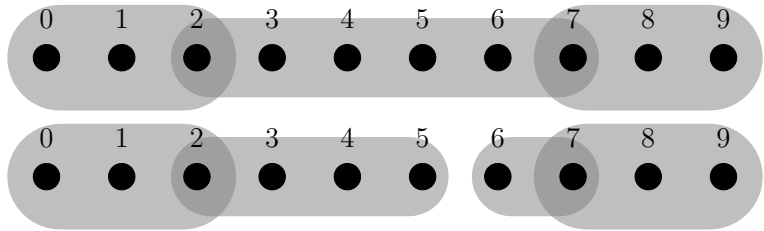
\end{Exe}

\begin{Obs}\label{rem:J}
	Let $\mathbf{J}$ be the matrix with all entries equal to $1$ and consider the  $n \times n$ block matrix given by $\mathbf{M}=\left[\begin{array}{cc}
	\mathbf{0}_{p} & \mathbf{J}_{p\times q}\\
	\mathbf{J}_{q \times p} & \mathbf{0}_{q}\\
	\end{array}\right],$ where $p+q=n.$  $\mathbf{M}$ has only two non zero eigenvalues which are given by $\sqrt{pq}$ and $-\sqrt{pq}.$
\end{Obs}

\begin{Teo}
	Let $\h$ be a hypergraph, $e \in E(\h)$ an edge and $e=e'\cup e''$ a division of this edge. Then $$|\E(\h)- \E(\h \triangleleft e)| \leq 2\sqrt{|e'||e''|}.$$
\end{Teo}

\begin{proof}
	First, notice that $$|\E(\h)- \E(\h \triangleleft e)| \leq  \E(\A(\h)-\A(\h \triangleleft e)).$$
	
	We can assume that $e=\{1,2, \ldots r\}$, where  $r=|e'|+|e''|.$  Now, also notice that
	
	$$\A(\h)-\A(\h \triangleleft e) = \left[\begin{array}{ccc}
	\mathbf{0}_{|e'|} & \mathbf{J}_{|e'|\times |e''|} & 0\\
	\mathbf{J}_{|e''| \times |e'|} & \mathbf{0}_{|e''|} & 0\\
	0 & 0 & 0\\
	\end{array}\right].$$
	
	From Remark \ref{rem:J}, we have that $ \E(\A(\h)-\A(\h \triangleleft e)) = 2\sqrt{|e'||e''|}.$
\end{proof}

\begin{Lem}
	Let $\h$ be a hypergraph and $e\in E(\h)$ an isolated edge (that is, the intersection of $e$ with any other edge is empty), then  $|\E(\h)- \E(\h \triangleleft e)|  = 2.$
\end{Lem}

\begin{proof}
	We can assume $\h=\h' \cup e$, where $e=\{1,2, \ldots r\}$ and $r=|e'|+|e''|.$ We have that
	
	$$\A(\h)= \left[\begin{array}{cc}
	(\mathbf{J}-\mathbf{I})_{|e| \times |e|} & \mathbf{0}\\
	\mathbf{0} & \A(\h')\\
	\end{array}\right] \quad and \quad \A(\h \triangleleft e) =\left[\begin{array}{ccc}
	(\mathbf{J}-\mathbf{I})_{|e'| \times |e'|} & \mathbf{0} & \mathbf{0}\\
	\mathbf{0} & (\mathbf{J}-\mathbf{I})_{|e''| \times |e''|} & \mathbf{0}\\
	\mathbf{0} & \mathbf{0} & \A(\h')\\
	\end{array}\right].$$
	
	Therefore
	
	$$\E(\h) = \E(\h') +2|e|-2 \quad and \quad \E(\h \triangleleft e) = \E(\h') +2|e'|-2 +2|e''|-2 .$$
	
	So
	
	$$\E(\h)- \E(\h \triangleleft e) = 2|e|-2 -2(|e'|+|e''|)+4=2.$$
	
\end{proof}

\begin{Obs}
	Let $\h$ be a hypergraph and $e\in E(\h)$ an isolated edge, then the equality $|\E(\h)- \E(\h \triangleleft e)|  = 2\sqrt{|e'||e''|}$ holds if, and only if, $|e|=2.$
\end{Obs}

\begin{Def}	Let $\h=(V,E)$ be a connected hypergraph and $F \subset E$ a subset of edges. The ordered triple $(F, F', F'')$ is said to be a \textit{weak cut}, if $\h\triangleleft (F, F', F'')$ is disconnected, and for no proper subset $P \subset F$, the multi-hypergraph $\h\triangleleft P$ is disconnected, considering any way to divide the edges of $P$.	
\end{Def}

For our next result we need the following theorem from \cite{energy-edge-deletion}.

\begin{Teo}\label{TeoX}\cite{energy-edge-deletion}
    For a real and symmetric partitioned matrix $\C=\left[ \begin{matrix} \M & \mathbf{X}\\
	\mathbf{Y} & \B \end{matrix}\right]$ where both $\M$ and $\B$ are square matrices, we have $$\sum_j |\lambda_j(\M)| + \sum_j |\lambda_j(\B)| \leq \sum_j |\lambda_j(\C)|.$$
	Moreover, equality holds if and only if there exist unitary matrices $\mathbf{U}$ and $\mathbf{V}$ such that $\left[ \begin{matrix} \mathbf{UM} & \mathbf{UX}\\
	\mathbf{VY} & \mathbf{VB} \end{matrix}\right] $ is positive semi-definite.
\end{Teo}

\begin{Teo}
	Let $\h=(V,E)$ be a connected hypergraph. If $F \subset E$ is a weak cut, then $\E(\h\triangleleft F) \leq \E(\h)$.
\end{Teo}
\begin{proof}
	Notice that $\h\triangleleft F$ is the disjoint union of two multi-hypergraphs, lets say $\h\triangleleft F=\h_1\cup\h_2$. So we have that
	$$\A(\h\triangleleft F)=\left[ \begin{matrix}\A(\h_1) & \mathbf{0}\\
	\mathbf{0} & \A(\h_1)\end{matrix}\right]\quad \textrm{ and }\quad  \A(\h)=\left[ \begin{matrix}\A(\h_1) & \mathbf{X}\\
	\mathbf{X}^T & \A(\h_2)\end{matrix}\right],$$
	where $\mathbf{X}=(x_{ij})$ is a matrix of order $|V(\h_1)|\times|V(\h_2)|$ and $x_{ij}$ is the number of edges in $F$ that contain the vertices $i\in V(\h_1)$ e $j\in V(\h_2)$.
	
	Therefore from Theorem \ref{TeoX} , we have that $\E(\h)\geq \E(\h_1)+\E(\h_2) = \E(\h\triangleleft F).$
\end{proof}

\begin{Lem}[Exercise 2 of Section 7.1, \cite{HornJohnson}]\label{LemSemiDef}	
	A positive	semidefinite matrix has a zero entry on its main diagonal if and only if the entire row and
	column to which that entry belongs is zero.
\end{Lem}

\begin{Teo}
	Let $\h=(V,E)$ be a connected hypergraph, $e \in E$ an edge and $v \in e$ a vertex. If $(e,e-\{v\},\{v\})$ is a weak cut, then $\E(\h\triangleleft(e,e-\{v\},\{v\})) < \E(\h)$.
\end{Teo}
\begin{proof}
	Notice that $\h\triangleleft (e,e-\{v\},\{v\})=\h_1\cup\h_2$. So we have that
	$$\A(\h\triangleleft(e,e-\{v\},\{v\}))=\left[ \begin{matrix}\A(\h_1) & \mathbf{0}\\
	\mathbf{0} & \A(\h_1)\end{matrix}\right]\quad \textrm{ and }\quad  \A(\h)=\left[ \begin{matrix}\A(\h_1) & \mathbf{X}\\
	\mathbf{X}^T & \A(\h_2)\end{matrix}\right],$$
	with $x_{ij} = 1$ if $i\in e-\{v\}$ and $j=v$, and with  $x_{ij}=0$ otherwise. We can label the vertices in such way that the first row of $\mathbf{X}$ is the only non null row.
	
	Supposing that $\E(\h\triangleleft(e,e-\{v\},\{v\})) = \E(\h)$, from Theorem \ref{TeoX} we know that exist orthogonal matrices $\mathbf{U}$ and $\mathbf{W}$, such that
	
	$$\left[ \begin{matrix}\mathbf{U}\A(\h_1) & \mathbf{U}\mathbf{X}\\	\mathbf{W}\mathbf{X}^T & \mathbf{W}\A(\h_1)\end{matrix}\right],$$
	is a positive semidefined matrix. So we have that $(\mathbf{U}\mathbf{X})^T=\mathbf{W}\mathbf{X}^T$, because of the structure of $\mathbf{X}$ we have that $\mathbf{W} = \left[\begin{matrix}\alpha& \mathbf{0}\\\mathbf{0}& \mathbf{W}_1\end{matrix}\right]$, with $|\alpha| = 1$ and $\mathbf{W}_1$ is a orthogonal matrix. So, we can write $\A(\h_1) = \left[\begin{matrix}0& \mathbf{y}^T\\\mathbf{y}& \mathbf{A}_1\end{matrix}\right]$, therefore $\mathbf{W}\A(\h_1) = \left[\begin{matrix}0& \alpha\mathbf{y}^T\\\mathbf{W}_1\mathbf{y}& \mathbf{W}_1\mathbf{A}_1\end{matrix}\right]$. From Lemma \ref{LemSemiDef}, we have that $\y=\mathbf{0}$, but this implies that $v$ has no neighbor in $\h_2$ what contradicts the fact that $\h$ is connected.	
\end{proof}

\begin{Def}
	A hypergraph $\mathcal{T}$ is a hypertree if $\mathcal{T}$ is connected and has no cycles.
\end{Def}

For a hypertree $\mathcal{T}$ we notice that:
\begin{enumerate} 	
    \item $\mathcal{T}$ is a linear hypergraph, since if there exist two edges $e_1$ e $e_2$ that have two (or more) common vertices $u$ e $v$, then we can construct the cycle  $ue_1ve_2u$.

    \item for any edge $e$ of a hypertree, if $v \in e$, then $(e,e-\{v\},v)$ is a weak cut.
\end{enumerate}

\begin{Cor}
	Let $\mathcal{T}$ be a hypertree. If $e \in E$ is an edge and $v \in e$ is a vertex, then $\E(\mathcal{T}\triangleleft(e,e-\{v\},\{v\})) < \E(\mathcal{T})$.	
\end{Cor}

\begin{Obs}
	Let $\h=(V,E)$ be a connected hypergraph and $e \in E$ an edge. If $v \in e$ has degree $1$, then $(e,e-\{v\},v)$ is a weak cut.
\end{Obs}

\begin{Cor}
	Let $\h=(V,E)$ be a connected hypergraph and $e \in E$ an edge. If $v \in e$ has degree $1$, then  $\E(\h\triangleleft(e,e-\{v\},\{v\})) < \E(\h)$.	
\end{Cor}

\begin{Exe} As it happens in $\h-e$, when we divide edges we can also have $\E(\h) > \E(\h \triangleleft e)$ and $\E(\h) < \E(\h \triangleleft e)$.

\begin{enumerate}
	\item
	$\E(\h) > \E(\h \triangleleft e)$, by doing any division at any edge:\\
	Let $\mathcal{K}_5^{[4]}$ be the complete $4$-uniform hypergraph with $5$ vertices and observe that $\E(\mathcal{K}_5^{[4]})=24.$ 	Let $e \in E$. Without loss of generality, we can assume that $e=\{1,2,3,4\}.$ We can divide $e$ in two different ways.\\
	
	Case 1: If $e'=\{1,2\}$ e $e"=\{3,4\}$, then $\E(\mathcal{K}_5^{[4]} \triangleleft e)=20,8924.$\newline
	
	Case 2: If $e'=\{1\}$ e $e"=\{2,3,4\}$, then $\E(\mathcal{K}_5^{[4]} \triangleleft e)=21,7438.$

\item
	 $\E(\h) < \E(\h \triangleleft e)$, by doing any division at any edge:\\
	
	Let $\h=(V,E)$ be the hypergraph give by $V=\{1,2,3,4,5,6,7,8\}$ and $$E=\left\lbrace \begin{matrix}1357,1358,1367,1368,1457,1458,1467,1468,\\2357,2358,2367,2368,2457,2458,2467,2468\end{matrix}\right\rbrace_.$$
	
	We have that  $\E(\h)=48$. Now, let $e \in E$ be any edge. Without loss of generality, we can assume that $e=\{1,3,5,7\}.$ We can divide $e$ in two different ways.\\

	Case 1: If $e'=\{1,3\}$ e $e"=\{5,7\}$, then $\E(\h)=48.3294$.\newline
	
	Case 2: If $e'=\{1\}$ e $e"=\{3,5,7\}$, then $\E(\h)=48,4723$.

	\item
	$\E(\h)=\E(\h \triangleleft e)$ by doing a particular division at a particular edge:
	
	Example \ref{ExAresta} (3) also works on this case. Notice that the edge $e$ has cardinality $2$ and in this case we have that $\A(\h \triangleleft e) = \A(\h-e).$
\end{enumerate}	
\end{Exe}

\section{Sharp bounds for energy}\label{bounds}
In this section we will obtain bounds for the adjacency energy of a hypergraph. These bounds will be computed as functions of important and well known spectral and structural parameters.

\begin{Def}
	Let $\h$ be a hypergraph. Its \textit{Zagreb index} is defined as the sum of the squares of the degrees of its vertices. More precisely $$Z(\h) = \sum_{v \in V(\h)}d(v)^2.$$ This is an important parameter in graph theory, having chemistry applications, see the survey \cite{energy-Zagreb}.
\end{Def}

We will start by presenting some upper bounds for the energy of a hypergraph

\begin{Teo} \label{cota-sup1}
	Let $\h=(V,E)$ be a hypergraph with rank $r$. If $|E|=m$  and $|V|=n$, then
	$$\E(\h) \leq \sqrt{n(r-1)Z(\h)} \leq \sqrt{nm(r^2-r)\Delta}.$$
	Equality holds if, and only if, $\h$ is a $2$-graph with isolated edges and no isolated vertices or if $\h$ is an edgeless hypergraph.
\end{Teo}

\begin{proof}
	Let $\lambda_1,\cdots,\lambda_n$ be the eigenvalues of $\A(\h)$. Then,
	$$\sum_{i=1}^n \lambda_i^2 = \mathrm{Tr}(\A^2) = \sum_{i=1}^n\sum_{j=1}^n d(\{i,j\})^2\stackrel{(*)}{\leq} \sum_{i=1}^n\left( d(i)\sum_{j=1}^n d(\{i,j\})\right) \stackrel{(**)}{\leq} (r-1)\sum_{i=1}^nd(i)^2 .$$
	
	Notice that the equality $(*)$ holds only on hypergraphs with the following property: If two vertices $u$ e $v$ are neighbors then they belong to the same edges. But this only happens on hypergraphs made of isolated edges (and possibly some isolated vertices).	Also notice that the equality $(**)$  holds only on uniform hypergraphs. Then, $$\sum_{i=1}^n \lambda_i^2 \leq (r-1)Z(\h).$$

	Therefore,	$$\E(\h)^2 = \left( \sum_{i=1}^n |\lambda_i|\right)^2  \stackrel{(cs)}{\leq} n\left(\sum_{i=1}^n \lambda_i^2 \right) \leq n(r-1)Z(\h).$$
	
	By Cauchy Schwarz inequality, we know that the equality $(cs)$ only is true when all the eigenvalues have the same absolute values. For uniform hypegraphs made of isolated edges, this only happens when there is no isolated vertex and the edges have size $2$.
	
	Since,	$$Z(\h) = \sum_{i=1}^nd(i)^2 \leq \sum_{i=1}^n\Delta d(i) \leq\Delta rm$$
	
	And this equality occurs only on regulars hypergraphs, 	we conclude that
	$$\E(\h) \leq \sqrt{n(r-1)Z(\h)} \leq \sqrt{nm(r^2-r)\Delta}$$
	
	With all the equalities happening only if $\h$ is a $2$-graph with isolated edges and no isolated vertices or if $\h$ has no edges.
\end{proof}

\begin{Teo} \label{cota-sup2}
	Let $\h=(V,E)$ be a hypergraph with rank $r$ and $|V|=n$, then
	$$\E(\h) \leq \lambda_1 + \sqrt{(n-1)[(r-1)Z(\h)-\lambda_1^2]}.$$
	Equality holds if, and only if, $\h$ is a $2$-graph with isolated edges and no isolated vertices or if $\h$ is an edgeless hypergraph.
\end{Teo}

\begin{proof}
	From Theorem \ref{cota-sup1} we have that
	$$\lambda_1^2 + \sum_{i=2}^n \lambda_i^2 = \sum_{i=1}^n \lambda_i^2 \leq (r-1)Z(\h).$$
	
	With equality holding only on uniform hypergraphs made of isolated edges and possibly some isolated vertices.
	
	Therefore,
	$$(\E(\h) - \lambda_1)^2=(\sum_{i=2}^n |\lambda_i|)^2  \stackrel{(cs)}{\leq} (n-1)\sum_{i=2}^n \lambda_i^2 \leq  (n-1)[(r-1)Z(\h)-\lambda_1^2].$$
	
	Notice that the equality $(cs)$ only is true when $|\lambda_2|=|\lambda_3|= \ldots = |\lambda_n|$.
	For uniform hypergraphs made of isolated edges, this can only occur when all the edges have size 2 and there is no isolated vertices.
\end{proof}

\begin{Lem}
    Let $\h$ be a hypergraph with rank $r$ and co-rank $s$, then  $$(s-1)\delta \leq \lambda_1 \leq (r-1)\Delta.$$
\end{Lem}
	
\begin{proof}
    Let $\x=(x_i)$ be the non negative eigenvector associated to the eigenvalue $\lambda_1$. Taking vertices $v,u \in V$ such that $x_v$ has maximum value and $x_u$ has minimum value, we have that

    $$\lambda_1 x_v = (\A\x)_v = \sum_{e \in E_{[v]}}x\left(e-\{v\}\right)$$
    therefore
    $$\lambda_1 = \sum_{e \in E_{[v]}}\frac{x\left(e-\{v\}\right)}{x_v} \leq \Delta\frac{(r-1)x_v}{x_v} = (r-1)\Delta.$$

    The other inequality follows similarly, using $x_u$.
\end{proof}

\begin{Cor}
	Let $\h=(V,E)$ be a  $r$-uniform and $d$-regular hypergraph, then
	$$\E(\h) \leq (r-1)d + \sqrt{(n-1)[(r-1)nd^2 -(r-1)^2d^2]}.$$
\end{Cor}

Now we will obtain some lower bounds for the adjacency energy.

\begin{Lem}\label{lema-cota1}
	Let $\h=(V,E)$ be a hypergraph with $|V|=n$, then
	$$\E(\h)^2 \geq  2\sum_{i=1}^n \lambda_i^2.$$
	Equality holds only if $\h$ has at most one $2$-uniform complete bipartite connected component and possibly isolated vertices, or if $\h$ has no edges.
\end{Lem}

\begin{proof}
	First, notice that
	$$0 = \left( \sum_{i=1}^n \lambda_i\right)^2 =  \sum_{i=1}^n \lambda_i^2 + 2\sum_{i<j}\lambda_i\lambda_j.$$
	Therefore,
	$$\sum_{i=1}^n \lambda_i^2 = -2\sum_{i<j}\lambda_i\lambda_j, \quad \Rightarrow \quad \sum_{i<j}|\lambda_i\lambda_j| \stackrel{(1)}{\geq} \left| \sum_{i<j}\lambda_i\lambda_j \right|  = \frac{1}{2}\sum_{i=1}^n \lambda_i^2.$$
	Then,
	$$\E(\h)^2 = \left( \sum_{i=1}^n |\lambda_i|\right)^2 = \sum_{i=1}^n |\lambda_i|^2 + 2\sum_{i<j}|\lambda_i\lambda_j| \stackrel{(2)}{\geq}  2\sum_{i=1}^n \lambda_i^2.$$
	
    Notice that equality $(1)$ and consequently $(2)$ occur only if every non zero eigenvalue have the same signal. If $\h$ has at least three non zero eigenvalues (not necessarily distincts) that is impossible to happen. Therefore $\h$ must have at most one 2-uniform complete bipartite connected component and possibly isolated vertices.
\end{proof}

In what follows, we present three lower bounds for the energy of a hypergraph, based on different parameters.

\begin{Teo}\label{cota-inf1}
	Let $\h=(V,E)$ be a hypergraph with co-rank $s$, $|V|=n$ and average degree $d(\h)$ then
	$$\E(\h) \geq \sqrt{2n(s-1)d(\h)}.$$
	Equality holds if, and only if, $\h$ is the $2$-uniform hypergraph with at most one complete bipartite connected component.
\end{Teo}

\begin{proof}	
	It suffices to note that
	$$\E(\h)^2\geq2\sum_{i=1}^n\lambda^2_i=2\sum_{i=1}^n\sum_{j=1}^nd(i,j)^2\stackrel{(3)}{\geq}2\sum_{i=1}^n\sum_{j=1}^nd(i,j)\geq2n(s-1)d(\h).$$
\end{proof}

\begin{Teo}\label{cota-inf2}
	Let $\h=(V,E)$ be a hypergraph with co-rank $s$ and $|V|=n$, then
	$$\E(\h) \geq \sqrt{\frac{2(s-1)^2}{n}Z(\h)}.$$
	Equality holds if, and only if, $\h$ is the $2$-uniform hypergraph with only one edge and no isolated vertices or if $\h$ has no edges.
\end{Teo}

\begin{proof}
	It suffices to note that
	$$\E(\h)^2\geq2\sum_{i=1}^n\sum_{j=1}^nd(i,j)^2\stackrel{(4)}{\geq}2\sum_{i=1}^n\left( \frac{1}{n}\left( \sum_{j=1}^nd(i,j)\right)^2 \right) \geq\frac{2(s-1)^2}{n}Z(\h).$$
	
	Notice that equality $(4)$ holds only if $d(i,j)$ is constant for each $i$. The only complete bipartite graph with this property is the edge.
\end{proof}

\begin{Obs}\label{obs-cota2}
	Let $\h$ be a hypergraph with $n$ vertices and co-rank $s$. We can compare the bounds provided by these two theorems. We define $b(\h) = \sqrt{2n(s-1)d(\h)}$ e $B(\h) = \sqrt{\frac{2(s-1)^2}{n}Z(\h)}$. It is easy to see that:
	
	\begin{enumerate}
		\item If $\frac{\delta^2}{\Delta}\geq\frac{n}{s-1}$, then $b(\h) \leq \sqrt{2n(s-1)\Delta} \leq \sqrt{2(s-1)^2\delta^2} \leq B(\h).$
		
		\item If $\frac{\Delta^2}{\delta}\leq\frac{n}{s-1}$, then $B(\h) \leq \sqrt{2(s-1)^2\Delta^2} \leq \sqrt{2n(s-1)\delta} \leq b(\h).$
		
		\item In particular, if $\h$ is $d$-regular, then:
		\begin{enumerate}
			\item If $d=\frac{n}{s-1}$, then $b(\h)=B(\h)$.
			\item If $d>\frac{n}{s-1}$, then $b(\h)<B(\h)$.
			\item If $d<\frac{n}{s-1}$, then $b(\h)>B(\h)$.
		\end{enumerate}
		
	\end{enumerate}
\end{Obs}

\begin{Teo}\label{cota-inf3}
	Let $\h$ be a hypergraph, then $$\E(\h) \geq \sqrt{\frac{2n}{n-1}\lambda_1^2}.$$
	Equality holds if, and only if, $\h$ is the $2$-uniform hypergraph with one edge and two vertices or if $\h$ has no edges.
\end{Teo}

\begin{proof}
	Notice that $$\sum_{i=1}^n \lambda_i^2 = \lambda_1^2 + \sum_{i=2}^n \lambda_i^2 \stackrel{(cs)}{\geq} \lambda_1^2 +\frac{1}{n-1} (\sum_{i=2}^n \lambda_i)^2 =  \lambda_1^2 + \frac{(-\lambda_1)^2}{n-1} = \frac{n}{n-1}\lambda_1^2.$$
	
	Therefore,
	$$ \E(\h)^2 \geq  2\sum_{i=1}^n\lambda_i^2 \geq \frac{2n}{n-1}\lambda_1^2.$$
	
	Equality $(cs)$ only is true when $|\lambda_2|=|\lambda_3|= \ldots = |\lambda_n|$. For complete  bipartite graphs, this only holds when $\h=K_2.$
\end{proof}

Note that  the proofs of the previous theorems leads immediately to distinct low bounds  of $\sum_{i=1}^n\lambda_i^2$. More specifically, we have the following lemma, which will be useful later.

\begin{Lem} \label{obs-cota-inf}
Let $\h$ be a hypergraph with co-rank $s$ and $|V|=n.$
	\begin{enumerate}
		\item
		$\sum_{i=1}^n\lambda_i^2 \geq n(s-1)d(\h).$
		
		\item
		$\sum_{i=1}^n\lambda_i^2 \geq \frac{(s-1)^2}{n}Z(\h).$
		
		\item
		$\sum_{i=1}^n\lambda_i^2 \geq  \frac{n}{n-1}\lambda_1^2.$
	\end{enumerate}
	
\end{Lem}

\begin{Lem} \label{lema-cota2}
	Let $\h=(V,E)$ be a hypergraph with $|V|=n$, then
	$$\E(\h)^2 \geq  \sum_{i=1}^n \lambda_i^2 + n(n-1)|\det(\A)|^{\frac{2}{n}}.$$
	Equality holds if $|\lambda_i| = |\lambda_j|$ for every $i \neq j$, that is, if $\h$ is the graph with one edge and two vertices .
\end{Lem}

\begin{proof}
	First, notice that
	$$\frac{1}{n(n-1)} \sum_{i \neq j}|\lambda_i \lambda_j|  \stackrel{(1)}{\geq} \prod_{i \neq j} |\lambda_i \lambda_j|^\frac{1}{n(n-1)} = \prod_{i=1}^n |\lambda_i|^\frac{2}{n} = |\det(\A)|^\frac{2}{n} $$
	Therefore,
	$$\E(\h)^2 = ( \sum_{i=1}^n |\lambda_i|)^2 =  \sum_{i=1}^n \lambda_i^2 +  \sum_{i \neq j} |\lambda_i \lambda_j| \geq \sum_{i=1}^n \lambda_i^2 + n(n-1)|\det(\A)|^{\frac{2}{n}}. $$
	
	Where $(1)$ holds because the arithmetic mean of non negative numbers is greater or equal than the geometric mean. That is:
	$$\frac{x_1+ \ldots + x_n}{n} \geq (x_1 \ldots x_n)^\frac{1}{n}$$
	and equality holds if, and only if, $x_1= \ldots = x_n.$
\end{proof}

From these lemmas, we establish other lower bounds for the energy of a hypergraph, depending on the determinant of its adjacency matrix.

\begin{Teo} \label{cota-inf4}
	Let $\h=(V,E)$ be a hypergraph with co-rank $s$ and $|V|=n$, then
	$$\E(\h) \geq \sqrt{n(s-1) d(\h) + n(n-1)|\det(\A)|^{\frac{2}{n}}}.$$
\end{Teo}

\begin{proof}
	Use Lemma  \ref{lema-cota2} and Remark \ref{obs-cota-inf}  $(1)$.
\end{proof}

\begin{Teo} \label{cota-inf5}
	Let $\h=(V,E)$ be a hypergraph with co-rank $s$ and $|V|=n$, then
	$$\E(\h) \geq \sqrt{\frac{(s-1)^2}{n}Z(\h) + n(n-1)|\det(\A)|^{\frac{2}{n}}}.$$
\end{Teo}

\begin{proof}
	Use Lemma \ref{lema-cota2} and Remark \ref{obs-cota-inf}  $(2)$.
\end{proof}

\begin{Teo}\label{cota-inf6}
	Let $\h=(V,E)$ be a hypergraph, then
	$$\E(\h) \geq \sqrt{\frac{n}{n-1}\lambda_1^2 + n(n-1)|\det(\A)|^{\frac{2}{n}}}.$$
\end{Teo}

\begin{proof}
	Use Lemma  \ref{lema-cota2} and Remark \ref{obs-cota-inf}  $(3)$.
\end{proof}

Note that the bounds obtained in Theorems \ref{cota-inf4} and \ref{cota-inf5} can also be compared, following the same cases listed in Remark \ref{obs-cota2}.

Finally, the next remark allow to compare the bounds of the last three theorems with the bounds of the three preceding theorems.

\begin{Obs} Regarding the bounds of Lemmas \ref{lema-cota1} and \ref{lema-cota2}, we have four cases:

\begin{enumerate}
		
		\item If $\h$ is a hypergraph that has zero as an eigenvalue, we have that the bound of Lemma  \ref{lema-cota1} is sharper than the bound of Lemma  \ref{lema-cota2} (consequently the bounds from Theorems \ref{cota-inf1}, \ref{cota-inf2} and \ref{cota-inf3} are better than the respective bounds from Theorems \ref{cota-inf4}, \ref{cota-inf5} and \ref{cota-inf6}).

		\item If $\g$ is a graph with $|\det(\A(\g)| \geq 1$ then the bound from Lemma \ref{lema-cota2} is better than the boudn from Lemma \ref{lema-cota1}.

			Indeed, we can assume that $\g$ has no isolated vertices.
			Saying that \ref{lema-cota2} is better than Lemma  \ref{lema-cota1} means that
			$$\E(\h)^2 \geq  \sum_{i=1}^n \lambda_i^2 + n(n-1)|\det(\A)|^{\frac{2}{n}} \geq  2\sum_{i=1}^n \lambda_i^2.$$
			And this is true if, and only if,
			$$ n(n-1)|\det(\A)|^{\frac{2}{n}} \geq  \sum_{i=1}^n \lambda_i^2.$$
			and in fact
			$$  n(n-1)|\det(\A)|^{\frac{2}{n}} \geq n(n-1) = 2\frac{n(n-1)}{2} \geq 2m =  \sum_{i=1}^n \lambda_i^2.$$

		\item We have that $(2)$ does not hold for hypergraphs in general. For exemple: Let $\h=(V,E)$ be a hypergraph with vertices $V= \left\{1,2,3,4,5,6,7,8\right\}$ and edges $$E=\left\{1234, 1237, 1238, 1256, 126, 1278, 3456, 3478, 5678\right\},$$ we have $|\det(\A(\h))|=252$ and Lemma \ref{lema-cota1} is sharper than Lemma  \ref{lema-cota2}, because :
		$$ \sum_{i=1}^n \lambda_i^2 >253 > 224 > n(n-1)|\det(\A(\h))|^{\frac{2}{n}}$$

		\item Let $\h$ be a hypergraph with rank $r$ and $|\det(\A(\h)| \geq 1.$ If $n \geq (r-1)\Delta^2+1$, the the bound of Lemma \ref{lema-cota2} is better than Lemma \ref{lema-cota1}.

			Indeed, notice that $$\sum_{i=1}^n \lambda_i^2 \leq (r-1)n\Delta^2$$.
			
			Therefore $$ n(n-1)|\det(\A)|^{\frac{2}{n}} \geq n(n-1) \geq  n(r-1)\Delta^2 \geq \sum_{i=1}^n \lambda_i^2 .$$	
	\end{enumerate}
	
\end{Obs}

\section{Conclusion}
In this paper, we made contributions to spectral hypergraph theory. More precisely, we define the energy of a hypergraph as the energy of its adjacency matrix, and study its properties. We obtain which hypergraphs have the highest and lowest energy within the class of hyperstars. By obtaining spectral properties of operations sum and product of hypergraphs, we prove that the energy of a hypergraph is never an odd number. We study how the hypergraph energy varies when we remove a vertex or an edge from it, and we further define an edge division operation and analyze how an edge division impacts the energy of a hypergraph. Our main results are the determination of bounds for the energy. These bounds are functions of well known parameters, such as maximum and minimum degree, Zagreb index and spectral radius.

We end this paper presenting some open problems about energies of hypergraphs.

\begin{enumerate}
	\item In \cite{Kaue-lap}, the authors define the signless Laplacian matrix for general hypergraphs, so we believe that many results proven in \cite{Kaue-energia} can be generalized to non-uniform hypergraphs using some of the techniques that we used here.
	
	\item We have determined which hyperstar with $t$ vertices have  highest and lowest energy. An interesting problem is to determine in other classes, which is the hypergraph with the highest and lowest energy.
\end{enumerate}

\section*{Acknowledgments}

This article is part of the PhD thesis of Lucas L. Portugal, supervised by Renata R. Del-Vecchio. Lucas L. Portugal acknowledges the financial support provided  by CAPES and Renata R. Del-Vecchio acknowledges the financial support  by  CNPq grant 306262/2019-3.
This research took place while Vilmar Trevisan visited the ‘Dipartimento di Matematica e Applicazioni’, University of Naples ‘Federico II’, Italy and this author acknowledges the financial support provided by the hosting University, and by CAPES-Print 88887.467572/2019-00, Brazil, as well as the partial support of CNPq grants 409746/2016-9 and 303334/2016-9, and FAPERGS grant PqG 17/2551-0001.

\bibliographystyle{acm}
\bibliography{Bibliografia}
\end{document}